\documentclass{amsart}[12pt]

\usepackage{stmaryrd,mathrsfs,amsmath,amsfonts,amssymb,graphicx,multirow}
\usepackage[plainpages=false]{hyperref}
\usepackage{graphicx}
\usepackage{comment}

\usepackage{amsthm}
\usepackage{graphics}
\usepackage{amsmath}
\usepackage{amssymb}
\usepackage[dvips]{epsfig}
\usepackage{psfrag}
\usepackage{bbm}
\usepackage{float}
\usepackage{latexsym}

\RequirePackage{color}
\theoremstyle{plain}

\numberwithin{equation}{section}

\newtheorem{theorem}{Theorem}[section]
\newtheorem{lemma}[theorem]{Lemma}

\newtheorem{proposition}[theorem]{Proposition}
\newtheorem{corollary}[theorem]{Corollary}

\theoremstyle{remark}

\newcommand{\lam}{\lambda}
\newcommand{\R}{{\mathbb R}}

\newcommand{\Vol}{\text{Vol}}
\newcommand{\cal}{\mathcal}

\begin{document}
\title{On existence of the prescribing $k$-curvature of the Einstein tensor}
\author{Leyang Bo}
\address{Leyang Bo: School of Mathematical Sciences, Zhejiang University, Hangzhou 310027, China.}
\email{bo\underline{ }leyang@126.com}

\author{Weimin Sheng}
\address{Weimin Sheng: School of Mathematical Sciences, Zhejiang University, Hangzhou 310027, China.}
\email{weimins@zju.edu.cn}

\thanks{The authors were supported by NSFC, grant no. 11571304.}
\keywords{Einstein tensor, $k$-curvature, existence, compactness}
\subjclass[2010]{53C21, 35J60 }

\maketitle
%%%%%%%%%%%%%%%%%%%%%
\begin{abstract}
In this paper, we study the problem of conformally deforming a metric on a $3$-dimensional manifold $M^3$ such that its $k$-curvature equals to a prescribed function, where the $k$-curvature is defined by the $k$-th elementary symmetric function of the eigenvalues of the Einstein tensor, $1\le k\le 3$. We prove the solvability of the problem and the compactness of the solution sets on manifolds when $k=2$ and $3$, provided the conformal class admits a negative $k$-admissible metric with respect to the Einstein tensor.
\end{abstract}
%%%%%%%%%%%%%%%%%%%%%%%%%
\baselineskip16pt
\parskip3pt

\section{Introduction}
Let $(M^n,g)$ be an $n$-dimensional compact Riemannian manifold with or without boundary, $n\ge 3$. Let $Ric$ and $R$ be the Ricci tensor and the scalar curvature, respectively. Then the Einstein tensor is defined by
\begin{equation*}
 E_g=\frac{1}{n-2}(Ric-\dfrac{R}{2}g),
 \end{equation*}
which can be also viewed as a special case of the following modified Schouten tensor with a parameter $\tau$ that was introduced by Gursky and Viaclovsky \cite{GV1}, and A. Li and Y.-Y. Li \cite{LL2} independently:
 \begin{equation*}
 A_g^\tau=\frac{1}{n-2}(Ric-\dfrac{\tau R}{2(n-1)}g),
 \end{equation*}
where $\tau\in \R$ and the Einstein tensor is just the case $\tau=n-1$.

Einstein tensor plays a key role in general relativity, and was extensively studied by many researchers. In this paper, we focus on its property in the conformal class.

Let $\lam(A_g^\tau)=(\lam_1,\cdots,\lam_n)$ denote the eigenvalues of $A_g^\tau$ with respect to $g$. We also define the $k$-curvature of $\lam(A_g^\tau)$ as
\[
\sigma_k(\lam)=\Sigma_{i_1<\cdots<i_k}\lam_{i_1}\cdots\lam_{i_k},
\]
the $k$-th elementary symmetric polynomial, and
\[
\Gamma_k^+=\{\lam\in \R^n\, |\, \sigma_j(\lam)>0\,\,  {\text{for}}\, \,  j=1, \cdots, k\}
\]
the corresponding open, convex cone in $\R^n$.

Define $\Gamma^-_k=\{\lam \in \R^n|-\lam\in \Gamma_k^+\}$.  We call $g$ is $k$-admissible if $\lam(g^{-1}A_g^\tau)\in\Gamma^+_k$ or negative $k$-admissible if $\lam(g^{-1}A_g^\tau)\in\Gamma^-_k$.

With the conformal transformation $\tilde g=u^{\frac{4}{n-2}}g$, the modified Schouten tensor changes by
\[
A_{\tilde g}^\tau=-\dfrac{2}{n-2}u^{-1}\nabla^2 u-\dfrac{2(1-\tau)}{(n-2)^2}u^{-1}(\Delta u)g+\dfrac{2n}{(n-2)^2}u^{-2}du\otimes du-\dfrac{2}{(n-2)^2}u^{-2}|\nabla u|^2g+A_g^\tau.
\]
When $\tau=n-1$, we have
\[
E_{\tilde g}=-\dfrac{2}{n-2}u^{-1}\nabla^2 u+\dfrac{2}{(n-2)}u^{-1}(\Delta u)g+\dfrac{2n}{(n-2)^2}u^{-2}du\otimes du-\dfrac{2}{(n-2)^2}u^{-2}|\nabla u|^2g+E_g.
\]

If $\tilde g=e^{-2u}g$, we have
\[
A_{\tilde g}^\tau= A^\tau_g+\nabla^2 u+\frac{1-\tau}{n-2}\Delta u+du\otimes du-\frac{2-\tau}{2}|\nabla u|^2g.
\]
In particular, when $\tau=n-1$, we have
\[
E_{\tilde g}=E_g+\nabla^2 u-\Delta u+du\otimes du+\frac{n-3}{2}|\nabla u|^2g.
\]

Similar to the $k$-Yamabe problem, the $k$-curvature of the Einstein tensor is defined by
\begin{equation}
\sigma_k(E_g)=\sigma_k(\lam(g^{-1}E_g)).
\end{equation}
It is natural to ask: can we find a $k$-admissible (or negative $k$-admissible) metric in the conformal class of $g$ with constant $k$-curvature of the Einstein tensor?  This problem is equivalent to find a solution to  the following equation:
\begin{equation}
\sigma_k(E_g+\nabla^2 u-\Delta u+du\otimes du+\frac{n-3}{2}|\nabla u|^2g)=1.
\end{equation}
or
\begin{equation}\label{Einstein}
\sigma_k(-\nabla^2 u+\Delta u-du\otimes du-\frac{n-3}{2}|\nabla u|^2g-E_g)=1.
\end{equation}

The equation \eqref{Einstein} is elliptic, which in some sense corresponds to the $k$-Yamabe equation in the positive cone, i.e
\begin{equation}\label{k-Yamabe}
\sigma_k(A^\tau_g+\nabla^2 u+\frac{1-\tau}{n-2}\Delta u+du\otimes du-\frac{2-\tau}{2}|\nabla u|^2g)=1
\end{equation}
with $\tau=1$.

The $k$-Yamabe problem have been extensively studied in the last decade. When $k=1$, equations \eqref{k-Yamabe} and \eqref{Einstein} become to the classical Yamabe equation. The existence of its solution has been solved by Yamabe \cite{Y}, Trudinger \cite{T}, Aubin \cite{A} and Schoen \cite{S}. The answer for the compactness of the solution set is positive when the dimension $n\le 24$ and negative when $n\ge 25$ (\cite{KMS, BM}). When $\tau=1$, it was initially studied by Viaclovsky \cite{V1}.
The existence of the solution to the equation was solved for $k=2$, $n\ge 4$ in \cite{CGY1, STW}, for $k\ge \frac{n}{2}$ by \cite{GV2, TW1, TW2}, for $2< k<n/ 2$ and $(M, g)$ being locally conformally flat case by \cite{GW1, LL1}. The compactness of the solution set was proved for $k=2$, $n=4$ by \cite{CGY1}, for $k\ge \frac{n}{2}$ by \cite{GV2, TW1, TW2, LN}, and for $2< k<n/ 2$ and $(M, g)$ being locally conformally flat case by \cite{ LL2}.

For general $\tau$, the case $\tau\le 1$ and $\tau\ge n-1$ is somehow meaningful since in these cases the equation is elliptic, for $\tau\neq 1$, equation \eqref{k-Yamabe} usually does not have the variational structure even the manifold is locally conformally flat, which requires some other ways to approach this problem. In \cite{LS, SZ}, the authors studied these cases and they give a positive answer to the problems on existence and compactness for $\tau<1$ in the negative cone and $\tau>n-1$ in the positive cone using a parabolic flow argument and proved the priori estimates are exponentially decayed. In \cite{SZ}, the authors also give a positive answer to long time existence of the flow for $\tau<1$ in the positive cone and $\tau>n-1$ in the negative cone on locally conformally flat manifolds. But the convergence of the flow is still unknown even on locally conformally flat manifold.

In this paper we mainly study the $k$-curvature equation of Einstein tensor ($\tau=n-1$) in the negative cone. This corresponds to the $k$-curvature of the Schouten tensor being in the positive cone.
We only deal with the three dimensional case since the admissibility of the Einstein tensor in the cases of $k=2$ and $3$ imply the non-negativity of Ricci tensor. By use of the idea in \cite{LN} and \cite{GV2} we can get the $C^0$ estimate of the solutions. For the higher dimensional case,   the non-negativity of Ricci tensor can not be obtained from this point.

Our main result is
\begin{theorem}\label{main thm}
Let $(M^3, g)$ be a $3$-dimensional closed Riemannian manifold and $k=2$, or $k=3$. Assume
\begin{enumerate}
  \item[(1)] $g$ is negative $k$-admissible with respect to the Einstein tensor, and
  \item[(2)] $(M^3, g)$ is not conformally equivalent to the standard sphere.
\end{enumerate}
Then for any given smooth positive function $h\in C^\infty(M)$, there exist a solution $u\in C^\infty(M)$ of
$\sigma_k(-\lam(E_{\tilde g}))=h(x)$, where $\tilde{g}=e^{-2u}g$,
and the set of all such solutions is compact in the $C^m$-topology for any $m>0$.
\end{theorem}

The rest of the paper is organized as follows. In section 2 we introduce the Liouville theorem and a Riemannian version of Hawking's singularity theorem in relativity. In section 3 we give the deformation of equation \eqref{Einstein}. In section 4 we give some local estimates of the solutions to the deformation equation, and then by contradiction argument, we show that there is at most one blow-up point in the manifold and establish the explicit blow-up speed around the blow-up point. Then using the Bishop-Gromov volume comparison theorem we get the contradiction, and finish the proof of Theorem \ref{main thm}.

\section{Preliminary}
In this section, we first introduce some basic properties of the elementary symmetric functions and Newton transformation.

The $k$-th Newton transformation associated with a real symmetric matrix $A$ is defined as follows:
$$T_k(A)=\sigma_k(A)I-\sigma_{k-1}(A)A+\cdots+(-1)^kA^k.$$
and we have
$$T_k(A)^i_j=\dfrac{1}{k!}\delta^{i_1\cdots i_k i}_{j_1\cdots j_k j}A_{i_1 j_1}\cdots A_{i_k j_k},$$
where $\delta^{i_1\cdots i_k i}_{j_1\cdots j_k j}$ is the generalized Kronecker delta symbol.
Note $\sigma_k(A)=\dfrac{1}{k!}\delta^{i_1\cdots i_k}_{j_1\cdots j_k}A_{i_1 j_1}\cdots A_{i_k j_k}$, we have
$T_{k-1}(A)^i_j=\dfrac{\partial \sigma_k(A)}{\partial A_{ij}}$.

Sometimes we also use the eigenvalue version of the $k$-th elementary symmetric function,
$\lam(A)=(\lam_1,\cdots,\lam_n)$ and $\sigma_k(\lam)=\Sigma_{i_1<\cdots<i_k}\lam_{i_1}\cdots\lam_{i_k}$.
We write $\sigma_0(\lam)=1$, $\sigma_k(\lam)=0$ for $k>n$, and denote $\sigma_{k;i}(\lam)=\sigma_k(\lam)|_{\lam_i=0}$, $\sigma_{k;ij}(\lam)=\sigma_k(\lam)|_{\lam_i,\lam_j=0}$.
We also denote $\partial_i \sigma_k=\dfrac{\partial\sigma_k}{\partial \lam_i}$, $\partial_i \partial_j \sigma_k(\lam)=\dfrac{\partial^2\sigma_k}{\partial \lam_i\partial \lam_j}$. Then $\partial_i \sigma_k(\lam)=\sigma_{k-1;i}(\lam)$, $\partial_i \partial_j \sigma_k(\lam)=\sigma_{k-2;ij}(\lam)$.

Now we introduce some basic properties for the $k$-th elementary symmetric function and Newton transformations which will be used frequently in this paper. One can find these facts in many literature, see \cite{STW} for example:
\begin{proposition}\label{basic property}
Let $\lam(A)\in \Gamma_k^+$ with $\lam_1\ge\cdots\ge\lam_n$. Then
\begin{description}
  \item[(i)] $T_{k-1}(A)^i_jA^j_i=k\sigma_k(A)$, or equivalently, $\sum_i\lam_i\sigma_{k-1;i}(\lam)=k\sigma_k(\lam)$,
  \item[(ii)] $trT_{k-1}(A)=(n-k+1)\sigma_{k-1}(A)$, or equivalently, $\sum_i\sigma_{k-1;i}(\lam)=(n-k+1)\sigma_{k-1}(\lam)$,
  \item[(iii)] $\sigma_k(\lam)=\lam_i\sigma_{k-1;i}(\lam)+\sigma_{k;i}(\lam)$, for each $i=1, \cdots, n$.
  \item[(iv)] $T_{k-1}(A)$ is positive definite, or equivalently, $\sigma_{k-1;n}\ge\cdots\ge\sigma_{k-1;1}>0$,
  we also have $\sigma_{k-2;ij}>0$ for each $i\neq j$.
\end{description}
\end{proposition}

We also have the following:
\begin{corollary}
Let $(M,g)$ be a $3$-dimensional Riemannian manifold and $x\in M$. $-E_g$ is 2-admissible at $x$, then
 $$Ric_g(x)\ge 0.$$
\end{corollary}
\begin{proof}
A direct calculation gives that $T_1(-E_g)=Ric_g$, by Proposition \ref{basic property}(iv), the Ricci tensor is positive definite.
\end{proof}

Now consider the equation
\begin{equation}\label{modi exp Yamabe}
\sigma_k(\lambda(-E_{g_u}))=u^{p-(n+2)/(n-2)}
\end{equation}
in the case $-\infty<p\le (n+2)/(n-2)$ and $n\ge 3$.
We have a Liouville type theorem:
\begin{theorem}(Theorem 1.1 of \cite{LL2})\label{Liouville}
For $n\ge 3$,  assume that $u\in C^2(\mathbb{R}^n)$ is a superharmonic solution of \eqref{modi exp Yamabe} for some $-\infty<p\le (n+2)/(n-2)$. Then either $u\equiv \text{constant}$ or $p=(n+2)/(n-2)$
and
\begin{equation}
u(x)\equiv \left(\dfrac{a}{1+cb^2|x-\bar x|^2}\right)^{\frac{n-2}{2}}, \, \, x\in \mathbb{R}^n
\end{equation}
for some fixed $\bar x\in\mathbb{R}^n$ and for some positive constants $a$ and $b$ satisfying $\sigma_k(\lambda(2b^2a^{-2}I))=1$, where $c=\frac{(n-2)}{(-2n+2+n\tau)}$.
\end{theorem}

Modifying some constants in Theorem 1.3 of \cite{LL2}, and noticing that $\sigma_k$ satisfies the condition of $f$ in the theorem, we can prove Theorem \ref{Liouville}. We omit the proof here.

Next we introduce a Riemmanian version of Hawking's singularity theorem which will play an important role in the blow-up analysis.

\begin{proposition}(see \cite{LN} or \cite{book1})
Let $(N^n,g)$ be a complete smooth Riemannian manifold with smooth boundary $\partial N$. If $Ric_{g}\geq -(n-1)\alpha^2$ for some $\alpha \geq 0$ and if the mean curvature $H$ of $\partial N$ with respect to its inward pointing normal satisfies $H \geq (n-1)c_0 \geq (n-1)\alpha$. Then
\begin{align}
d_g (x,\partial N)\leqslant U(\alpha,c_0) ,\, \, \forall x \in N.
\end{align}
where $d_g$ denotes the distance function induced by $g$ and
\begin{align}
U(\alpha,c_0)=\left\{
             \begin{array}{ll}
               \frac{1}{c_0} ,& \hbox{ if $\alpha =0$;} \\
               \frac{1}{\alpha}\coth^{-1}(\frac{c_0}{\alpha}), & \hbox{ if $\alpha > 0$;}
             \end{array}
           \right.
\end{align}
\end{proposition}

The following lemma is about the symmetric matrix, which is needed in our proof of the main theorem. The proof can be found in \cite{LN}.
\begin{lemma}(\cite{LN} Lemma A.1)
For an $n\times n$ real symmetric matrix $M$, let $\lam_1(M),\cdots,\lam_n(M)$ denote its eigenvalues. There exists a constant $C(n)>0$ such that $\forall \epsilon>0$ and any two symmetric matrices $M$ and $\tilde M$ satisfying $|M-\tilde M|<\epsilon$, there holds for some permutation $\sigma=\sigma(M,\tilde M)$ that
$$\sum_{i=1}^n|\lam_i(M)-\lam_{\sigma(i)}(\tilde M)|<C(n)\epsilon.$$

\end{lemma}

\section{Deformation of the equation}
In this section we deform equation \eqref{Einstein} to a equation which has an unique solution. Consider
\begin{equation}\label{deform}
\begin{split}
&\sigma_k^{\frac{1}{k}}(\lambda(g^{-1}[\lambda_k(1-\varphi(t))g-\varphi(t)E_g-\nabla^2 u+\Delta u g-du\otimes du+\frac{3-n}{2}|\nabla u|^2g]))\\
&=\varphi(t)h(x)e^{-2u}+(1-t)(\int_M e^{-(n+1)u})^{\frac{2}{n+1}}
\end{split}
\end{equation}
where $\varphi \in C^1[0,1]$ satisfies $0\le\varphi(t)\le 1$, $\varphi(0)=0$, $\varphi(t)=1$ for $t\ge \frac{1}{2}$; and  $\lambda_k=(C_n^k)^{-\frac{1}{k}}$.
This deformation is valid for any $k$ and $n$.

At $t=1$,  \eqref{deform} becomes \eqref{Einstein}; while $t=0$, it turns to
\begin{equation*}
\begin{split}
&\sigma_k^{\frac{1}{k}}(\lambda(g^{-1}[\lambda_k g-\nabla^2 u+\Delta u g-du\otimes du+\frac{3-n}{2}|\nabla u|^2g]))\\
&=(\int_M e^{-(n+1)u})^{\frac{2}{n+1}}.
\end{split}
\end{equation*}
We can show that this equation has the unique solution $u(x)\equiv 0$. It is easy to see that $u\equiv 0$ is a solution. The uniqueness can be shown as follows: let $x_0$ be the maximum point of $u$ on $M$. At this point we have $\left.\nabla u \right|_{x_0}=0$, and $\left.\nabla^2 u \right|_{x_0}$ is negative semi-definite and then $-\nabla^2 u+\Delta u g$ is negative semi-definite. Then at $x_0$ we get
\begin{equation*}
\begin{split}
\lambda_k &=\sigma_k^{\frac{1}{k}}(\lambda(g^{-1}\cdot g))\\
&\ge \sigma_k^{\frac{1}{k}}(\lambda(g^{-1}[\lambda_k g-\nabla^2 u+\Delta u g+du\otimes du+\frac{3-n}{2}|\nabla u|^2g]))\\
&=\left(\int_M e^{-(n+1)u}\right)^{\frac{2}{n+1}}
\end{split}
\end{equation*}
Similarly, at the minimum point of $u$, we have $\lambda_k\le(\int_M e^{-(n+1)u})^{\frac{2}{n+1}}$. Therefore, $\lambda_k =(\int_M e^{-(n+1)u})^{\frac{2}{n+1}}.$

On the other hand, by Newton-Maclaurin inequality, we have,
\begin{equation*}
\begin{split}
\lambda_k &= \sigma_k^{\frac{1}{k}}(\lambda(g^{-1}[\lambda_k g-\nabla^2 u+\Delta u g-du\otimes du+\frac{3-n}{2}|\nabla u|^2g]))\\
&\le \frac{1}{n}\sigma_1(\lambda(g^{-1}[\lambda_k g-\nabla^2 u+\Delta u g+du\otimes du+\frac{3-n}{2}|\nabla u|^2g]))\\
&= \frac{1}{n}\left((n-1)\Delta u -\dfrac{n(n-3)+2}{2}\left|\nabla u\right|^2+n\lambda_k\right).
\end{split}
\end{equation*}
Then we get
$$(\frac{n}{2}-1)\int_M \left|\nabla u\right|^2\le\int_M \Delta u=0$$
thus $u\equiv\text{constant}=0$.

Now we define the operator as in \cite{GV2, HS1}
\begin{equation*}
\begin{split}
\Psi_t&=\sigma_k^{\frac{1}{k}}(\lambda(g^{-1}[\lambda_k(1-\varphi(t))g-\varphi(t)E_g-\nabla^2 u+\Delta u g-du\otimes du+\frac{3-n}{2}|\nabla u|^2g]))\\
&-\varphi(t)h(x)e^{-2u}-(1-t)(\int_M e^{-(n+1)u})^{\frac{2}{n+1}}.
\end{split}
\end{equation*}
When $t=0$, $\Psi_0[u]=0$ has unique solution $u\equiv 0$ and the linearization of $\Psi_0$ at $u\equiv 0$ is invertible.

We define the Leray-Schauder degree $deg(\Psi_t, {\mathcal{O}}, 0)$ as in \cite{L2, GV3}, where ${\cal{O}}=\{u\in C^{4,\alpha}(M): u \, \text{is}\,  k\text{-admissible}, dist(\lambda(A^\tau_g),\partial \Gamma^+_k)>\frac{1}{C}\}$. Then $deg(\Psi_0,{\cal{O}}_0,0)\neq 0$ at $t=0$. Of course we would like to use the homotopy-invariance of the degree to conclude  $deg(\Psi_t,{\cal{O}},0)$ is non-zero for some open set ${\cal{O}}\subset C^{4,\alpha}$.  To do this we need to establish a priori estimate for \eqref{deform}.
By \cite{HS2}, $C^1$ and $C^2$  estimates have been already done. Once we get the $L^\infty$ bound, by the standard Evans-Krylov theory, we have the $C^{2,\alpha}$ estimate and then the higher order estimates can be established.

Here we also note for $t\in [0, 1-\delta]$ and $\delta$ sufficiently small, the equation \eqref{deform} has $C^0$ estimate by the
work of He-Sheng \cite{HS2}. In fact this can be done by use of the $C^1$ and $C^2$ estimates in \cite{HS2} and a similar
argument in \cite{GW1}. So we just need to focus on the case $t = 1$.
%%%%%%%%%%%%%%%%%%%%%%%%%%

\section{The priori estimates}
In this section, we derive the $C^0$ bound by a contradiction argument.

In \cite{HS1, HS2}, the authors have already got the point-wise $C^1$ and $C^2$ estimate \eqref{4.3} with $g_u=e^{-2u}g$:
\begin{equation}\label{4.3}
\sup_M (|\nabla^2 u|+|\nabla u|^2)(x)\le C\left(1+e^{-2\inf_M u}\right).
\end{equation}

Here for convenience we use the form $g_u=u^{\frac{4}{n-2}}g$, substitute it to \eqref{4.3} we have
\begin{equation}
\sup\limits_M (|\nabla^2 \log u|+|\nabla \log u|^2)(x)\le C(1+\max\limits_M u^{\frac{4}{n-2}}).
\end{equation}
Now we only need to derive the $C^0$ bound.

Without loss of generality, we may assume that $\sigma_2(\lam(-E_{g_{can}}))=1$ on $S^n$, where $g_{can}$ is the standard metric on $S^n$.

First we'll show the upper bound of $u$ implies it's lower bound.
Assume at this moment we have established the estimate
$$\max u\le C.$$
This implies that
\begin{equation}
\sup\limits_M (|\nabla^2 \log u|+|\nabla \log u|^2)(x)\le C.
\end{equation}
%%%%%%%%%%%%%%%%%%%%

Now we'll show that :
$$\min_M u\ge \frac{1}{C}.$$

We argue it by contradiction. Suppose there exists a sequence $\{u_i\}$ such that
\begin{equation}\label{contr}
\min_M u_i\rightarrow 0.
\end{equation}
By definition, the metrics $g_i=u_i^{\frac{4}{n-2}}g$ satisfy
\begin{equation}\label{symsqYamabe}
\sigma_k(-\lam(E_{g_i}))=1~~-\lam(E_{g_i})\in\Gamma_k^+.
\end{equation}
Note here $g_i=(\frac{u_i}{u_1})^{\frac{4}{n-2}}g_1$, evaluating \eqref{symsqYamabe} at the maximum point $\bar x_i$ of $\frac{u_i}{u_1}$, since we have the relationship between $E_{g_i}$ and $E_{g_1}$:
\begin{equation*}
\begin{split}
E_{g_i}&=-\dfrac{2}{n-2}(\frac{u_i}{u_1})^{-1}\nabla^2 (\frac{u_i}{u_1})+\dfrac{2}{(n-2)}(\frac{u_i}{u_1})^{-1}(\Delta (\frac{u_i}{u_1}))g_1+\dfrac{2n}{(n-2)^2}(\frac{u_i}{u_1})^{-2}d(\frac{u_i}{u_1})\otimes d(\frac{u_i}{u_1})\\
&-\dfrac{2}{(n-2)^2}(\frac{u_i}{u_1})^{-2}|\nabla (\frac{u_i}{u_1})|^2g_1+E_{g_1}
\end{split}
\end{equation*}
and $-\frac{2}{n-2}(\frac{u_i}{u_1})^{-1}\nabla^2 (\frac{u_i}{u_1})+\frac{2}{(n-2)}(\frac{u_i}{u_1})^{-1}(\Delta (\frac{u_i}{u_1}))g_1$ is positive definite at $\bar x_i$(by $\tau<1$),
we obtain
$$1\ge \left(\dfrac{u_i}{u_1}\right)^{-\frac{4}{n-2}}{\sigma_k}(-\lam(E_{g_1}(\bar x_i)))=\left(\dfrac{u_i}{u_1}\right)^{-\frac{4}{n-2}}$$
which implies $\max_M u_i\ge u_1(\bar x_i)\ge \min_M u_1$. On the other hand we have the Harneck inequality $\min_M u_i\ge \frac{1}{C}\max_M u_i$ which comes from the $C^1$ estimate. Combining these inequalities, we have $\min_M u_i\ge\frac{1}{C}\min_M u_1$ which contradicts \eqref{contr}. Since we have established the $L^\infty$, $C^1$ and $C^2$ estimates, the higher order estimate on $\log u$ follows from the standard Evans-Krylov's and Schauder's estimates.

Now it is sufficient to check $\max_M u\le C$ by the above argument.
We also argue it by contradiction. Assume $\exists \{u_i\}$ a sequence of smooth positive functions on $M$ such that
$g_i=u_i^{\frac{4}{n-2}}g$ satisfy \eqref{Einstein} but
$$u_i(x_i)=\max\limits_M u_i\rightarrow\infty ~~\text{as}~~ i\rightarrow\infty. $$
Suppose $x_\infty$ is a blow-up point, we have $x_i\rightarrow x_\infty$ in the metric topology induced by the initial metric on $M$.
%%%%%%%%%%%%%%%%%%%%
\subsection{The unique blow-up point}
Now we show the blow-up point is unique, in fact we can control the speed of blow-up of the sequence $u_i$,
\begin{equation}\label{blow-up}
u_i\le Cd_g(x,x_i)^{-\frac{n-2}{2}},~~\forall x\in M\backslash\{x_i\}.
\end{equation} where C is a constant independent of $i$.

This property was observed by Y.Y.Li-Nguyen \cite{LN}. Our proof follows the work in \cite{LN}.
The following lemma plays an important role to establish \eqref{blow-up}. It reveals some concentration property of the volume of a small neighborhood of the blow-up sequence $\{x_i\}$.

\begin{lemma}{\rm{(Lemma 3.1 of \cite{LN})}}
Assume for some $C_1\ge 1,K_i\rightarrow\infty$ and $y_i\in M$, that $u_i\rightarrow\infty$; set
$D_{K_i}=\{y\in M|d_g(y,y_i)\le K_i u_i(y_i)^{-\frac{2}{n-2}}\} $, and assume
\begin{equation}\label{$D_{K_i}$}
\sup\limits_{D_{K_i}}u_i\le C_1 u_i(y_i).
\end{equation}
Then for any $0<\mu<1$, there exists $K=K(C_1,\mu)$ such that for $i$ sufficiently large
$$\Vol_{g_i}(D_{K})\ge (1-\mu)\Vol_{g_i}(M)$$
\end{lemma}

\begin{proof}
This argument comes from Lemma 3.1 of \cite{LN}. We need a little modification here since the coefficients in the formula of conformal change of Einstein tensor $E_g$ is different to the Schouten tensor $A_g$ in \cite{LN}. We state the outline of the proof.
Let $p\in \mathbb{R}^n,a>0$ and $b=\frac{n-2}{2n-2-n\tau}$, define
\begin{equation}
U_{a,p;b}(x)=\left(\dfrac{2a}{1+ba^2|x-p|^2}\right)^{\frac{n-2}{2}}
\end{equation}
$$S^n=\{z=(z_1,\cdots,z_n)\in\mathbb{R}^{n+1}|z_1^2+\cdots+z_n^2=1\}.$$
Let $(x_1,\cdots,x_n)\in\mathbb{R}^n$ be the stereographic projection coordinates of $S^n$, then
$$g_{can}=|dz|^2=\left(\dfrac{2}{1+|x|^2}\right)^2|dx|^2.$$
Let $x=\sqrt b x'$ we then have
$$g_{can}=|dz|^2=\left(\dfrac{2}{1+b|x'|^2}\right)^2 b|dx'|^2=U_{1,0;b}^{\frac{4}{n-2}}(x')b|dx'|^2=U_{1,0}^{\frac{4}{n-2}}(x')g_{flat}$$
where $g_{flat}=|dx|^2$ is the standard Euclidean metric on $\mathbb{R}^n$.
It follows that
$$f(\lam(-E_{U_{a,p;b}^{\frac{4}{n-2}}g_{flat}}^\tau))=1.$$

Now we define a map from the tangent space to the manifold $\Phi_i:T_{y_i}(M,g)\rightarrow M$ by
$$\Phi_i(x)=\exp_{y_i}\dfrac{2x}{u_i(y_i)^{\frac{2}{n-2}}},$$
then we can get $f(\lam(A_{(u_i\circ\Phi_i)^{\frac{4}{n-2}}\Phi_i^{*}g}^\tau))=1$
and let
$$\tilde u_i(x)=\dfrac{2^{\frac{n-2}{2}}}{u_i(y_i)}u_i\circ\Phi_i(x), x\in \mathbb{R}^n.$$

By Lemma 3.1 of \cite{LN}, we konw $\tilde u_i$ subconverges  to some positive $\tilde u_*\in C^2(\mathbb{R}^n)$ and actually by the Liouville theorem in section 2, $\tilde u_*=U_{a_*,x_*;b}$ for some bounded $a_*>0$ and $x_*\in \mathbb{R}^n$.

Process as lemma 3.1 of \cite{LN}, this implies that
$\forall \epsilon>0,\exists R=R(\epsilon,C_1)>0$ such that
$$|{\text{Vol}}_{g_i}(\Phi_*(B(0,R)))-{\text{Vol}}(S^n)|\le C\epsilon^n$$ for some $C$ independent of $i$ and $\epsilon$.
Let $H_{g_i}$ be the mean curvature of $\partial\Phi_i(B(0,R))$ with respect to $g_i$, this also implies $H_{g_i}\ge \epsilon^{-1}$.

By the Hawking's lemma in section 2 and the above property, we see that
$${\text{diam}}_{g_i}(M\backslash\Phi(B(0,R)))\le C\epsilon.$$
this inequality together with Bishop's comparison theorem imply that
$${\text{Vol}}_{g_i}(M\backslash\Phi(B(0,R)))\le C\epsilon^n,$$
then ${\text{Vol}}_{g_i}(\Phi(B(0,R)))\ge (1-\epsilon^n)C$.

Now let $K=\frac{1}{2}R$ and $\mu=\epsilon^n$, then $D_{K_i}=\Phi(B(0,R)$, the proof is complected.
\end{proof}

With this lemma we can immediately get \eqref{blow-up} by the argument in \cite{LN}.

Besides, we can also get the similar higher order estimates as in step 2 of \cite{LN} :
\begin{equation}\label{gradient-blow-up}
|\nabla^k \log u_i(x)|\le C d_g(x,x_i)^{-k} ~~\text{for} x\neq x_i, k=1,2.
\end{equation}
and the similar convergence behavior as in step 3 of \cite{LN} :
 $$\lim\limits_{i\rightarrow\infty} u_i=u_\infty$$
with $u_\infty(x)\equiv 0 $ in $M\setminus\{x_\infty\}$. Unfortunately this limit function is difficult to analysis. In order to get some useful result we need to rescale it as follows: fix some point $p\in M\setminus\{x_\infty\}$ and let
$$v_i(x)=\dfrac{u_i(x)}{u_i(p)}$$
$v_i$ subconverges, for every $0<\alpha<1$, in $C^{1,\alpha}(M\setminus \{x_\infty\},g)$ to some positive function $v_\infty\in C^{1,1}(M\setminus \{x_\infty\},g)$ which satisfies $v_\infty(p)=1$ and
\begin{equation}\label{infty-gradient-blow-up}
|\nabla^k \log v_\infty(x)|\le C d_g(x,x_\infty)^{-k} ~~\text{for} x\neq x_i, k=1,2.
\end{equation}
Furthermore, $-E_{g_{v_\infty}}$ is on the boundary of the Garding cone $\Gamma^+_k$ in viscosity sense as $u_i(p)$ trends to 0.

%%%%%%%%%%%%%%%%%%%%%%
\subsection{The blow-up order}
In this section, we show that $v_{\infty}$ has an asymptotic behavior near $x_{\infty}$ of order $n-2$, i.e
\begin{equation}
\lim\limits_{x\rightarrow x_\infty} v_\infty d_g(x,x_\infty)=a
\end{equation}
where $a\in (0,+\infty)$ is a constant.

By \eqref{infty-gradient-blow-up}, we can immediately get
\begin{equation}
\limsup\limits_{x\rightarrow x_\infty} v_\infty d_g(x,x_\infty)=A<+\infty.
\end{equation}
We only need to show
\begin{equation}\label{a}
\liminf\limits_{x\rightarrow x_\infty} v_\infty d_g(x,x_\infty)=a
\end{equation}
and $A=a$.

As in \cite{LN}, note here $-E_{g_{v_\infty}} \in \partial\Gamma^+_k$. Then we have the super-harmonicity of $v_\infty$, i.e $\Delta_g v_\infty-\frac{(n-2)}{4(n-1)}R_g v_\infty\le 0$. Let $L_g=\Delta_g-\frac{n-2}{4(n-1)}R_g$ be the conformal Lapalacian, then $L_gv_\infty\le 0$.
We then have the following lemma.
\begin{lemma}(Lemma 3.3 of \cite{LN})\label{4.2}
Let $\Omega$ be an open neighborhood of a point $p\in M$. If $w$ is a nonnegative lower semi-continuous function in
$\Omega\setminus \{p\}$ and satisfies $L_g w\le 0$ in the viscosity sense in $\Omega\setminus \{p\}$, then
$$\lim_{r\rightarrow 0}r^{n-2}\min_{\partial B_g(p,r)}w<+\infty.$$
\end{lemma}
By Lemma \ref{4.2} we immediately know that $a$ is finite.

Now it remains to show $A=a$. We prove it by contradiction also. Assume $A>a$, then we can find a sequence $\{y_i\}$ such that for some $\epsilon>0$,
$$A+\epsilon\ge d_g(y_i,x_\infty)^{n-2}v_\infty(y_i)\ge a+2\epsilon,$$
where $x_\infty=\lim_{i\rightarrow \infty} y_i$.
Also, by \eqref{a}, we have
$$d_g(y_i,x_\infty)^{n-2}\min_{d_g(y,x_\infty)=d_g(y_i,x_\infty)}v_\infty(y)\le a+\epsilon.$$

Let $R_i=d_g(y_i,x_\infty)^{-1}$, define the exponential map $E_i: B_{\delta R_i}\subset T_{x_\infty}M\rightarrow M$ by
$$\Theta_i(y)=\text{exp}_{x_\infty}(R_i^{-1}y),$$
where $\delta$ is sufficiently small, and as in subsection 4.1 we have $R_i^2\Theta_i^*g$ converges on compact subsets to the standard Euclidean metric.
Now set $\hat v_i(y)=R_i^{2-n}v_\infty\circ E_i(y)$, then $\hat v_i\in C^{1,1}_{loc}(B_{\delta R_i}\setminus \{0\})$. By a direct computation we have
\begin{equation}\label{irradical}
\left\{
\begin{split}
&\lam(-E_{\hat v_i(y)^{\frac{4}{n-2}}R_i^2\Theta_i^*g})\in \partial \Gamma^+_k ~~\text{in}~~ B_{\delta R_i}\setminus\{0\} \\
 &\min_{\partial B_1} \hat v_{i}\le a+\epsilon ~~\text{and}~~ \max_{\partial B_1} \hat v_{i}\ge a+2\epsilon.
\end{split}
\right.
\end{equation}
Then by the estimates \eqref{gradient-blow-up} in subsection 4.1, a subsequence of $\hat v_i$ converges uniformly
to a limit $\hat v_*\in C^{1,1}_{loc}(\mathbb{R}^n\setminus \{0\})$ and satisfies in the viscosity sense
\begin{equation}
\lam\left(-E_{\hat v_*^{\frac{4}{n-2}}g_{flat}}\right) \in \partial\Gamma^+_k~~\text{in}~~\mathbb{R}^n\setminus\{0\}
\end{equation}
then we get $v_*$ is radially symmetric by the following lemma.
\begin{lemma}(cf. Theorem 1.18 in \cite{L1})
For $n\ge 3$, Let $\Gamma^+_k$ be the Garding cone, and let $u$ be a positive, locally Lipschitz viscosity solution of
\begin{equation}
\lam\left(-E_{u^{\frac{4}{n-2}}g_{flat}}\right) \in \partial\Gamma^+_k~~\text{in}~~\mathbb{R}^n\setminus\{0\}.
\end{equation}
Then $u$ is radially symmetric about the origin and $u'(r)\le 0$ for almost all $0<r<\infty$.
\end{lemma}
\begin{proof}

Note $$-E_{u^{\frac{4}{n-2}}g_{flat}}=\dfrac{2}{n-2}u^{-1}\nabla^2 u-\dfrac{2}{(n-2)}u^{-1}(\Delta u)g-\dfrac{2n}{(n-2)^2}u^{-2}du\otimes du+\dfrac{2}{(n-2)^2}u^{-2}|\nabla u|^2g.$$
The proof in Theorem 1.18 of \cite{L1} mainly depend on the ellipticity of the operator in the formula of
$A_{u^{\frac{4}{n-2}}g_{flat}}$, and in our case, the operator $\dfrac{2}{n-2}u^{-1}\nabla^2-\dfrac{2}{(n-2)}u^{-1}(\Delta )g$ is elliptic. The proof in Theorem 1.18 of \cite{L1} is still hold, we omit it here.
\end{proof}

This makes a contradiction to the second line of \eqref{irradical}, we then get $A=a$.

Now we only need to show that $a>0$. We do not use the condition $n=3$ and $k=2$ or $k=3$ untill this moment,  in the next lemma, we'll assume these conditions, and since $\Gamma_3^-\subset\Gamma_2^-$, we only need to prove the case $k=2$.

In normal coordinates at $x_i$, let $r=|x|$.
The following lemma is the key ingredient, which is a special version of Lemma 3.4 in \cite{LN}
\begin{lemma}\label{outside cone}
There exists some small $r_1>0$ depending only on $(M,g)$ such that for all $0<\delta<\frac{1}{4}$, the function
$v_{\delta}:=r^{-(1-2\delta)}e^r$ satisfies
\begin{equation}
\lam(E_{(v_{\delta})^{\frac{4}{n-2}}g})\in \mathbb{R}^n\setminus \Gamma_2^- ~~\text{in}~~\{0<r<r_1\}
\end{equation}
\end{lemma}
\begin{proof}
Let $b=1-2\delta$, when $n=3$, the Einstein tensor for $g_\delta=v_{\delta}^{\frac{4}{n-2}}g$ reads
\begin{equation*}
\begin{split}
E_{g_\delta}&=-\dfrac{2}{n-2}v_{\delta}^{-1}\nabla^2 v_{\delta}+\dfrac{2}{(n-2)}v_{\delta}^{-1}(\Delta v_{\delta})g+\dfrac{2n}{(n-2)^2}v_{\delta}^{-2}dv_{\delta}\otimes dv_{\delta}-\dfrac{2}{(n-2)^2}v_{\delta}^{-2}|\nabla v_{\delta}|^2g+E_g\\
&=-2v_{\delta}^{-1}\nabla^2 v_{\delta}+2v_{\delta}^{-1}(\Delta v_{\delta})g+6v_{\delta}^{-2}dv_{\delta}\otimes dv_{\delta}-2v_{\delta}^{-2}|\nabla v_{\delta}|^2g+E_g\\
&=D_1 I-D_2\dfrac{x}{r}\otimes\dfrac{x}{r}+D_3
\end{split}
\end{equation*}

where $I$ is the identity matrix and
\begin{equation*}
D_1=2\dfrac{v'_{\delta}}{v_{\delta}r}-2(\dfrac{v'_{\delta}}{v_{\delta}})^2+2\dfrac{v''_{\delta}}{v_{\delta}}
=2\dfrac{r-a}{r^2}-2\dfrac{(r-a)^2}{r^2}+2\dfrac{(r-a)^2+a}{r^2}
=\dfrac{2}{r}
\end{equation*}

\begin{equation*}
D_2=2v_{\delta}^{-1}(v''_{\delta}-\dfrac{v'_{\delta}}{r})-6\dfrac{v'^2_{\delta}}{v^2_{\delta}}
=\dfrac{4(1-a)a+2(4a-1)r-4r^2}{r^2}
\end{equation*}

\begin{equation*}
|D_3|\le C(1+rv^{-1}_{\delta}|v'_{\delta}|+r^2v^{-2}_{\delta}|v'_{\delta}|^2)\le C\le CrD_1
\end{equation*}

The eigenvalue of $D_1 I-D_2\dfrac{x}{r}\otimes\dfrac{x}{r}$ with respect to $I$ are $D_1-D_2, D_1,D_1$, we can apply Lemma 2.5 to see that the eigenvalues $\lam=(\lam_1,\lam_2,\lam_3)$ of $E_{g_\delta}$ with respect to $g_\delta$ satisfies
\begin{equation*}
|\lam_1-v^{-4}_{\delta}(D_1-D_2)|+|\lam_2-v^{-4}_{\delta}D_1|+|\lam_3-v^{-4}_{\delta}D_1|\le Cv^{-4}_{\delta}\le Crv^{-4}_{\delta}D_1
\end{equation*}

We have $\sigma_1(D_1-D_2, D_1,D_1)=3D_1-D_2=\dfrac{-4(1-a)(a-2r)+4r^2}{r^2}<0$ for $a<1$ and $r$ is sufficiently small.
Then $\sigma_2(D_1-D_2, D_1,D_1)=D_1(3D_1-2D_2)<0$.
This implies $(D_1-D_2, D_1,D_1)$ lying outside of $\bar\Gamma_2^-$(also $\bar\Gamma_3^-$ immediately). Thus $\lam(E_{g_\delta})$
lies outside of $\bar\Gamma_2^-$(also $\bar\Gamma_3^-$) since $r$ is sufficiently small.
\end{proof}

From $v_i(p)=1$ and \eqref{gradient-blow-up}, there exists some positive constant $C$ independent of $i$ and $\delta$ such that $v_i\ge\dfrac{1}{C}v_{\delta}$, on $\{r=r_1\}$. For some $K=K(\delta)>0$ large enough, let
$$\bar\gamma=\sup\left\{0<\gamma<\frac{1}{C}: v_i\ge \gamma v_{\delta} ~~\text{in}~~\{Ku_i(x_i)^{-\frac{4}{n-2}<r<r_1}\}\right\}.$$
By Lemma \ref{outside cone} and the comparison principle, there exist $\hat x_i$ with $|\hat x_i|=r_i$ such that
$$v_i(\hat x_i)=\bar\gamma v_{\delta}(\hat x_i).$$
Then follow the argument in \cite{LN}, we have $\bar\gamma=\frac{1}{C}$,
This shows $v_i\ge\dfrac{1}{C}v_{\delta}\ge \dfrac{1}{Cd_g(x,x_i)^{n-2-2\delta}}$ in $\{Ku_i(x_i)^{-\frac{4}{n-2}}<r<r_1\}$ for $i\ge N$ with $N$ sufficiently large.
When $i\rightarrow \infty$, we have
$$v_\infty\ge\dfrac{1}{C}v_{\delta}\ge \dfrac{1}{Cd_g(x,x_\infty)^{n-2-2\delta}} ~~\text{in}~~ \{0<r<r_1\}$$
for all sufficiently small $\delta>0$, finally this implies $a>0$.

Now note we already get the bound $C^{-1} d_g^{2-n}(x,x_\infty)\le v_\infty(x)\le C d_g^{2-n}(x,x_\infty)$,
also note the Ricci curvarure is semi-positive definite when $n=3$ and $k=2$, we can follow the work in \cite{LN,GV2,TW1} to show $(M\setminus{x_\infty},v_\infty^{\frac{4}{n-2}}g)$ is isomorphic to $(\mathbb{R}^n,g_{flat})$ by the Bishop-Gromov comparison theorem. Then by the consequence of Gursky-Viaclovsky\cite{GV2}, this implies that $(M,g)$ is conformally equivalent to the standard sphere, which is a contradiction to the initial hypothesis. This finishes our proof of the boundness of $u$.

\end{document}